\documentclass[11pt,oneside,leqno]{amsart}

\usepackage{amsxtra}
\usepackage{amsopn}
\usepackage{amsmath,amsthm,amssymb}
\usepackage{amscd}
\usepackage{amsfonts}
\usepackage{latexsym}
\usepackage{verbatim}

\theoremstyle{plain}
\newtheorem{theorem}{Theorem}[section]

\newtheorem{prop}[theorem]{Proposition}
\newtheorem{cor}[theorem]{Corollary}

\theoremstyle{definition}
\newtheorem{ex}[theorem]{Example}
\newtheorem{rem}[theorem]{Remark}
\newtheorem{definition}[theorem]{Definition}

\newcommand\C{{\mathbb C}}

\newcommand\R{{\mathbb R}}

\newcommand{\de}{\partial}
\newcommand{\debar}{\overline{\partial}}

\newcommand{\bismut}{\nabla^B}

\newcommand{\g}{\mathfrak{g}}

\newcommand{\ov}[1]{\overline{ #1}}

\newcommand\T{{\mathbb T}} 
\begin{document}
\title[Special Hermitian metrics and Lie groups]{Special Hermitian metrics and Lie groups}
\author{Nicola Enrietti and  Anna Fino}
\address{Dipartimento di Matematica G. Peano \\ Universit\`a di Torino\\
Via Carlo Alberto 10\\
10123 Torino\\ Italy} \email{nicola.enrietti@unito.it}
 \email{annamaria.fino@unito.it}
\thanks{This work was supported by the Projects MIUR ``Riemannian Metrics and Differentiable Manifolds'',
``Geometric Properties of Real and Complex Manifolds'' and by GNSAGA
of INdAM}
\begin{abstract} A Hermitian metric on a complex manifold
is called  strong K\"ahler with torsion (SKT)
if its fundamental $2$-form $\omega$ is
$\partial \overline \partial$-closed. We review some properties
of strong KT metrics  also in  relation with  symplectic  forms  taming complex structures. Starting from a $2n$-dimensional SKT Lie algebra $\mathfrak g$  {and using} a Hermitian flat connection on $\mathfrak g$ we construct a $4n$-dimensional SKT Lie algebra. We apply this method to
 some $4$-dimensional SKT  Lie algebras.  Moreover, we classify symplectic  forms  taming complex structures on $4$-dimensional Lie algebras.\end{abstract}
\maketitle

\section{Introduction}

A $J$-Hermitian metric $g$ on a complex manifold $(M,J)$ is called SKT (\emph{strong K\"ahler with torsion}) or \emph{pluriclosed}
 if the fundamental 2-form $\omega(\cdot,\cdot)=g(J\cdot,\cdot)$ satisfies
$$
\partial\ov \partial \omega=0
$$
(see for instance  \cite {GHR}).
For complex surfaces a  Hermitian metric satisfying the SKT condition is  {\emph{standard}} in the terminology of Gauduchon \cite{Gau2} and on a compact manifold a standard metric  can be found in the conformal
class of any given Hermitian metric. However,  the theory is completely different in higher dimensions.
The study of SKT  metrics  is strictly related to the study of the geometry of the Bismut connection. Indeed, any Hermitian manifold $(M, J,g)$ admits a unique connection $\nabla^B$ preserving $g$ and $J$ and such that the tensor
$$
c (X, Y, Z) = g(X, T^B(Y , Z ))
$$
is totally skew-symmetric, where by $T^B$ we denote the torsion of $\nabla^B$ (see \cite{Gau}).
This connection was introduced by Bismut in \cite{Bismut}  to prove a local index
formula for the Dolbeault operator for non-K\"ahler manifolds.
The torsion $3$-form $c$ is related to the fundamental form $\omega$  of $g$ by
$$
c(X, Y, Z)  =  - d \omega (JX, JY, JZ)
$$
and it is well known that  $
 \partial\ov \partial \omega=0$   is equivalent to $ dc=0$.

SKT metrics have a central role in type II string theory, in $2$-dimensional supersymmetric
$\sigma$-models (see \cite{GHR,strominger})  and they have also relations with generalized
K\"ahler geometry (see for instance
\cite{GHR,Gu,Hi2,AG,CG,FT}).
Indeed, by \cite{Gu,AG}   it follows  that a generalized K\"ahler
structure  on a  $2n$-dimensional manifold $M$ is equivalent to a pair of SKT structures  $(J_+, g)$ and $(J_-, g)$ such that  
$  d^c_+  \omega_+  = -d^c_-  \omega_- $, where $\omega_{\pm} ( \cdot, \cdot) = g( J_{\pm} \cdot, \cdot)$ are the
fundamental $2$-forms associated to the Hermitian structures
$(J_{\pm}, g)$ and $d^c_{\pm} = i (\overline \partial_ {\pm} -
\partial_{\pm})$. The closed $3$-form $ d^c_+  \omega_+$  is called the
{\it torsion} of the generalized K\"ahler structure and the   structure is said {\em untwisted} or {\em twisted} according to the fact  that  the cohomology class $[d^c_+  \omega_+] \in H^3 (M, \R)$ vanishes or not.  In particular, any K\"ahler metric $(J, g)$ determines  a generalized K\" ahler structure by setting $J_+ = J$ and $J_-
= \pm J$.

Recently, SKT metrics have been studied by many
authors. For instance, new simply-connected  compact  SKT examples have been
constructed by Swann in \cite{Sw} via the twist construction and SKT structures on $S^1$-bundles over almost contact manifolds have been studied in \cite{FFUV}. Moreover, in  \cite{FT2} it has been shown that the blow-up  of an SKT  manifold   at a point or along a compact submanifold  admits an SKT
metric.

For real Lie groups admitting left-invariant  SKT metrics there are  some   classification results in  dimension $4$, $6$ and $8$. More precisely,
$6$-dimensional  (resp. $8$-dimensional) SKT nilpotent Lie groups have been classified  in \cite{FPS}  (resp. in \cite{EFV}  and for a particular class in \cite{RT}) and  a classification of SKT
solvable Lie groups of dimension $4$ has been  obtained in \cite{MS}. 

General results {are  known also}  for nilmanifolds, i.e. compact quotients of simply connected nilpotent Lie groups  $G$ by discrete subgroups $\Gamma$. Indeed, in \cite{EFV} it has been proved that
if $(M=G/\Gamma,J)$ is a  nilmanifold  (not a torus) endowed with an invariant complex structure $J$ and if there exists a $J$-Hermitian {\rm SKT} metric $g$ on $M$, then  $G$ must be $2$-step nilpotent and $M$  is a  total space of a  principal holomorphic torus bundle over a torus.

No general restrictions  for the existence of  SKT  and generalized K\"ahler structures
  are known in the case of
solvmanifolds, i.e. compact quotients of solvable Lie
groups  by  discrete subgroups. A structure theorem by
\cite{Ha2} states that a solvmanifold carries a K\" ahler
structure if and only it is covered by a complex torus which has a
structure of a complex torus bundle over a complex torus.

As far as we know, the only known solvmanifolds carrying  a
generalized K\" ahler structure are the   Inoue surface of type ${\mathcal S}^0$ defined in \cite{In}  and  ${\mathbb T}^{2k}$-bundles over  the   Inoue surface of type ${\mathcal S}^0$   constructed in \cite{FT}. 

A  quaternionic analogous of K\"ahler manifolds   is given by  {\em
hyper-K\"ahler with torsion}  (shortly {\em HKT})  manifolds, that are hyper-Hermitian manifolds $(M^{4n}, J_1, J_2, J_3, h)$
admitting a hyper-Hermitian connection with totally skew-symmetric
torsion, i.e. for which the three Bismut connections associated to
the three Hermitian structures $(J_r, h)$, $r =1,2,3$  coincide.
This geometry was introduced by Howe and Papadopoulos \cite{HP} and
later studied for instance  in~\cite{GP, FG, BDV, BF, Sw}.
 In  \cite{BF} it was  shown that the
tangent Lie algebra of an HKT Lie algebra may admit an HKT
structure, 
constructing in this way   a family of new compact strong HKT manifolds.

In this paper we adapt the previous construction to the SKT and generalized K\"ahler  case.  Starting from a $2n$-dimensional SKT  (generalized K\"ahler)  Lie algebra $\mathfrak g$ and using a suitable  connection  on $\mathfrak g$  we construct a    $4n$-dimensional SKT  (generalized K\"ahler) Lie algebra. We apply the  previous procedure to
some of the $4$-dimensional SKT  Lie algebras, obtaining  in this way new SKT examples in dimension 8 and recovering the generalized K\"aher example found in \cite{FT}.

{The existence of an SKT metric $\omega$ on  a complex manifold $(M, J)$  such that $\de\omega=\debar\beta$ for a $\de$-closed (2,0)-form $\beta$   is equivalent to the existence of a symplectic form taming $J$ (\cite{EFV}).}

We recall that an almost complex structure $J$ on a compact $2n$-dimensional symplectic manifold $(M, \Omega)$ is  said to be \emph{tamed} by $\Omega$ if
$$
\Omega(X,JX)>0
$$
for any non-zero vector field $X$ on $M$.  When $J$ is a complex structure (i.e. $J$ is integrable) and $\Omega$ tames $J$, the pair $(\Omega, J)$ has been called a {\em Hermitian-symplectic structure} in \cite{ST}.
Although any symplectic structure always admits tamed almost complex structures,
it is still an open problem to find an example of a compact complex manifold admitting a
taming symplectic structure but no K\"ahler structures. From \cite{ST,LiZhang} there exist no compact examples in dimension $4$.
Moreover, the study of taming symplectic structures in dimension $4$ is related to a more general conjecture of Donaldson (see for instance  \cite{donaldson,weinkove,LiZhang}).

In \cite{EFV} some negative results for the existence of taming symplectic structures on compact quotients of Lie groups by discrete subgroups were obtained.
It was shown  that  if $M$ is
a nilmanifold (not a torus) endowed with an invariant complex structure $J$, then $(M, J)$ does
not admit any symplectic form taming  $J$.

{The taming symplectic structures are related to static solutions of
a new metric flow on complex manifolds (see  \cite{ST}).} Indeed, Streets and
Tian constructed an elliptic  flow using the Ricci tensor associated to the Bismut  connection instead of
the Levi-Civita connection, and it turns out that this flow preserves the SKT condition and that the existence of some particular type of static SKT metrics implies the existence
of a taming symplectic structure on the complex manifold (\cite{ST2}).
 Static SKT metrics on Lie groups
have been also recently studied in  \cite{En}.

In the last section of the paper we prove that a 4-dimensional Lie algebra $\g$ endowed with a complex structure $J$ admits a taming symplectic structure if and only if $(\g,J)$ admits a K\"ahler metric. Moreover, under this condition, every SKT metric induces a symplectic form taming $J$.

\medskip

\noindent \textbf{Acknowledgements.} This work has been partially
supported {by} Project MICINN (Spain) MTM2008-06540-C02-01/02,
Project MIUR ``Geometria Differenziale e Analisi Globale"  and
by GNSAGA of INdAM.
The second author wants to thank the organizers for the invitation and  their wonderful hospitality in Brno.

 \section{Preliminaries}

 Let $M$ be a $2n$-dimensional manifold.
We recall that an almost  complex structure $J$ on M  is
 integrable   if and only if
  the Nijenhuis tensor
$$N(X,Y)= J([X,Y]-[JX,JY])-([JX,Y]+[X,JY])$$
 vanishes for all vector fields  $X,Y$. In this case   $J$  is called a complex structure on $M$.

A  Riemannian metric $g$ on a complex manifold $(M, J)$ is said to
be  Hermitian if it is compatible with $J$, i.e. if  $g (JX, JY) = g
(X, Y)$ for any $X, Y$. In  \cite{Gau}, Gauduchon proved that if
$(M,J,g)$ is an Hermitian manifold, then there is a $1$-parameter family of
canonical Hermitian connections on $M$ characterized by  the properties
of their torsion tensor. In particular,
the {\em Bismut connection} is the unique connection $\nabla^B$  such that
$$
\nabla^B J = \nabla^B g = 0
$$
and its torsion  tensor
$$c (X, Y, Z) = g (X, T^B (Y, Z))$$ is totally skew-symmetric, where
$T^B$ is the torsion of  $\nabla^B$.  The geometry associated to the
Bismut connection is called KT geometry and when $c =0$ it coincides
with the usual K\"ahler geometry. 

\begin{definition} Let $(M, J,g)$ be a Hermitian manifold. 
If  the torsion 3-form $c$ of the Bismut connection  is  $d$-closed  or equivalently if $ \partial \overline \partial \omega=0$,  then   the Hermitian metric $g$ on   a
complex manifold $(M, J)$  is called {\em{strong K\"ahler with torsion}} (shortly SKT). 

\end{definition}
An interesting case is when $g$ is compatible with two complex structures $J_+$ and $J_-$.
 We recall the following
 
\begin{definition} A Riemannian  manifold  $(M,  g)$  is called
  {\em  generalized K\"ahler}   if it has a pair
 of  SKT  structures $(J_+, g)$  and $(J_-, g)$ for which $c_-= - c_+$,
 where $c_{\pm}$ denotes the torsion $3$-form  of the Bismut connection   associated to the  SKT structure $(J_{\pm},  g)$.
 \end{definition}

{G}eneralized K\"ahler structures were introduced  in \cite{GHR} and studied by M. Gualtieri in his
PhD thesis \cite{Gu} in the more general context of generalized
complex geometry, which contains complex and symplectic geometry as extremal special cases and shares important properties with them.

When $M$ is $4$-dimensional, by \cite{AG}  there are 
two classes of generalized K\"ahler structures, according to whether  the complex structures  $J_+$ and $J_-$  induce 
the same or different orientations on $M$. In  \cite{AG}  
compact  $4$-dimensional generalized K\"ahler manifolds $(M^4, J_{\pm}, g)$   for which  $J_+$ and  $J_-$
commute have been classified. 

In this paper we  consider  Lie algebras  endowed with SKT  structures  which induce  left-invariant SKT structures  on  the corresponding simply connected Lie groups.

  Let $\mathfrak g$ be a Lie algebra  with an  (integrable)  complex
structure $J$  and an inner product $g$ compatible with $J$. If the
associated K\"ahler form $\omega (X,Y)= g(JX,Y)$ satisfies
$d\omega=0$, where
$$
d \omega(X, Y, Z) = - \omega ([X, Y], Z) - \omega([Y, Z], X) -
\omega([Z, X], Y),
$$
for any $X, Y, Z \in {\mathfrak g}$,
 the Hermitian Lie algebra  $(\mathfrak g , J,g)$ is  K\"ahler.   Equivalently, $(\mathfrak g , J,g)$ is  K\"ahler
if and only if  $\nabla^{g} J =0$, where $\nabla^{g} $ is the Levi-Civita connection of $g$. If $\partial \overline \partial \omega =0$, the Hermitian Lie algebra  $(\mathfrak g , J,g)$ is  SKT.
We recall that for a  Lie  group $G$ with a
left-invariant  Hermitian structure {$(J,g)$} the Bismut connection $\nabla^B$ on $G$  is given by
the following equation
{\begin{equation} \label{Bismut}
\begin{array}{l}
g(\bismut_XY,Z) = \frac{1}{2} \{ g([X,Y]-[JX,JY],Z) - g([Y,Z]+[JY,JZ],X)\\
\phantom{g(\bismut_XY,Z) = }\, -g([X,Z]-[JX,JZ],Y) \}
\end{array}
\end{equation}}
for any $X, Y, Z \in {\mathfrak g}$ (see  \cite{DF}). 

If $G$ is nilpotent and admits a left-invariant SKT structure {$(J,  g)$}, then by \cite{EFV} $G$ has to be $2$-step nilpotent and $J$ has to preserve the center of $G$.  {Moreover, nilpotent Lie groups 
   cannot admit any left-invariant  generalized K\"ahler structure unless they are abelian \cite{Ca}.}
   
   In the solvable case there  is  a classification of SKT
 Lie groups of dimension $4$ \cite{MS}, but there are no  general results in higher dimensions.
Examples of  SKT and generalized K\"ahler solvable Lie groups admitting compact quotients 
have been shown  in \cite{FT, FT3}.

 \section{SKT structures on tangent Lie algebras}

  Let $\mathfrak g$ be a $2n$-dimensional Lie algebra   endowed with a   complex
structure $J$  and an inner product $g$ compatible with $J$.  Assume   that $D$ is a flat connection on
$\mathfrak g$ preserving the Hermitian structure, i.e.  such that $D g=0$ and $D J =0$.

  Consider the tangent Lie algebra $T_{D} \,  \g :=
{\mathfrak g} \ltimes_{D} {\R^{2n}}$ {endowed} with the Lie
bracket
\begin{equation}\label{t_D}  [(X_1,X_2), (Y_1,Y_2)]=([X_1,
Y_1] , D _{X_1} Y_2 - D_{Y_1} X_2 )
\end{equation}
and {with the} complex structure
\begin{equation}  \label{complexontg}
{\tilde J}(X_1, X_2)= (J X_1, J X_2).
\end{equation}

Since $D$ is flat, the Lie bracket \eqref{t_D} on $T_{D} \,  \g$ satisfies the Jacobi identity. The integrability of the complex
structure ${\tilde J}$ on $T_{D} \,  \g$ follows from
the fact that $J$ is integrable and parallel with respect
to $D$ (see \cite[Proposition~3.3]{BD}).

Let $\tilde g$  be the inner product on
$T_{D} \, \g$ induced by $g$ such that $( {\mathfrak g}, 0)$
and  $(0, {\mathfrak g})$ are orthogonal. Then $\tilde g$ is
compatible with ${\tilde J}$, that
is, $(T_D\,{\mathfrak g}, \tilde J, \tilde g)$ is a
Hermitian Lie algebra. 
 
In a similar way  as for  the HKT case (see \cite[Proposition 4.1]{BF}) we can prove the following

\begin{prop} \label{properties} Let  $({\mathfrak g}, J, g)$ be
a Hermitian  Lie algebra and  $D$  a flat connection such that $Dg=0$ and
$D J =0$. Then the Hermitian structure $(\tilde J, \tilde g)$ on $T_D \,  \g$ is   SKT if and only if  $(J, g)$ is  SKT on $\g$.
\end{prop}

 \begin{proof} The proof is already contained in \cite{BF}. Indeed,  by  a direct computation we have
that the Bismut connection $\tilde \nabla^B$ of the new  Hermitian 
structure $(\tilde J, \tilde g)$ on $T_{D} \,
{\mathfrak g} $ is related to the Bismut connection $\nabla^B$ of
the Hermitian structure $(J,  g)$ on ${\mathfrak g}$ by
\begin{equation} \label{exprtildenablaB}
 \tilde g  (\tilde \nabla^B_{(X_1, X_2)} (Y_1, Y_2), (Z_1, Z_2)) =  g(\nabla^B_{X_1} Y_1, Z_1) + g (D_{X_1} Y_2, Z_2),
\end{equation}
for any $X_i, Y_i, Z_i \in {\mathfrak g}$, $i = 1,2,3$. Therefore,
denoting  by $\tilde c$ and $c$ the torsion $3$-forms of the Bismut connections   on $T_{D} \, {\mathfrak g}$ and $\mathfrak g$,
respectively, we obtain
\begin{equation}\label{eqc}
\tilde c ((X_1, X_2), (Y_1, Y_2), (Z_1, Z_2)) = c (X_1, Y_1, Z_1),
\end{equation}
and
\begin{equation}\label{eqdc}
d\tilde c  \,  ((X_1, X_2), (Y_1, Y_2), (Z_1, Z_2), (W_1, W_2)) = dc
(X_1, Y_1, Z_1, W_1).
\end{equation}
This shows that the strong  condition is
preserved.
\end{proof}
 
\begin{rem}
 To construct the tangent Lie algebra we consider the flat connection $D$ as a representation of $\g$ on $\R^{2n}$. If we choose the adjoint representation {${\rm ad}$} of $\mathfrak g$ on  $\mathfrak g$ then the semidirect product  {$\g\ltimes_{{\rm ad}}\R^{2n}$} is the Lie algebra of the Lie group $TG$, the tangent bundle over $G$ \cite{BD}. In this case the conditions $DJ=Dg=0$ are satisfied if and only if $(\g,J)$ is a complex Lie algebra and the inner product $g$ is bi-invariant. Therefore this construction allows us to lift an invariant SKT structure $(\hat \g,\hat J)$ from a Lie group $G$ to its tangent bundle $TG$ if and only if $(G,\hat J)$ is a complex Lie group and $\hat g$ is a bi-invariant metric.
  \end{rem}

Since a generalized K\"ahler
structure  on a  $2n$-dimensional Lie algebra  $\mathfrak g$ is equivalent to a pair of  SKT  structures  $(J_+, g)$ and $(J_-, g)$,   such that  
$ c_+    = -  c_-$, {as a consequence of  the previous proposition we can prove} the following

 \begin{cor}\label{cor}
Let  $({\mathfrak g}, J_{\pm}, g)$ be
a generalized K\"ahler  Lie algebra and  $D$  a flat connection such that $Dg=0$ and
$D J_{\pm} =0$. Then,  $(\tilde J_{\pm}, \tilde g)$ on $T_D \,  \g$ is   generalized K\"ahler.
\end{cor}

\begin{proof}
We know that $\tilde J_-$ and $\tilde J_+$ are integrable and that the metric $\tilde g$ is still compatible with both complex structures. Moreover, if $({\mathfrak g}, J_{\pm}, g)$ is a generalized K\"ahler Lie algebra, then $c_+=-c_-$ and $dc_+=dc_-=0$. Therefore, {equations \eqref{eqc} and \eqref{eqdc} yield} $\tilde c_+=-\tilde c_-$ and $d\tilde c_+=d\tilde c_-=0$, i.e. $(T_D \,  \g,\tilde J_{\pm}, \tilde g)$ is a generalized K\" ahler  Lie algebra.
\end{proof}

For  Hermitian Lie algebras ${\mathfrak g}$ such that the commutator $[{\mathfrak g}, {\mathfrak g}]$ does not coincide with ${\mathfrak g}$ we can show that it is always possible to find an Hermitian flat connection $D$. 

\begin{prop}\label{connection}
Let $(\g,J,g)$ be a  2n-dimensional  SKT Lie algebra such that $\g^1=[\g,\g]\subsetneq\g$. Then $\g$ admits a flat connection $D$ such that $DJ=Dg=0$.
\end{prop}
\begin{proof}
Let $X\in\g\setminus\g^1$, and choose a basis $\{e_i\}$ of $\g$ such that $e_{2n}=X$. We define 
\[
\begin{cases}
D_{e_i}Y=0 \quad i=1,\dots,2n-1 \\
D_{e_{2n}}Y=JY.
\end{cases}
\]
It is easy to verify that $D$ is flat. Moreover, it {satisfies} the conditions
\[
g(D_XY,Z)=-g(Y,D_XZ) \qquad\ J(D_XY)=D_X(JY)
\]
since $g$ is Hermitian.
\end{proof}

\begin{rem}
Note that   the strict inclusion $[\g,\g]\subsetneq\g$  holds for every solvable Lie algebra $\g$.  {Therefore, by applying  the previous proposition,} we have that  every SKT solvable Lie algebra admits an SKT tangent Lie algebra.
\end{rem}

{Next we will}  apply the previous construction to   $4$-dimensional  SKT Lie algebras.  We recall that in the solvable case a  $4$-dimensional SKT   Lie group is unimodular if and only if it admits a compact quotient by a discrete subgroup (see \cite{MS}). 

By \cite{Ha3} a complex (non-K\"ahler) surface diffeomorphic to a
$4$-dimensional compact homogeneous manifold $X =\Theta \backslash L$, where $\Theta$ is a uniform
discrete subgroup of $L$,
and which does
not admit any K\" ahler structure is one of the following:\vskip.1truecm\noindent
a) Hopf surface;\vskip.1truecm\noindent
 b) Inoue surface of type ${\mathcal S}^0$;\vskip.1truecm\noindent
c) Inoue surface of type ${\mathcal S}^{\pm}$;\vskip.1truecm\noindent
d) primary Kodaira surface;\vskip.1truecm\noindent
e) secondary Kodaira surface;\vskip.1truecm\noindent
f) properly elliptic surface with first odd Betti number.\smallskip

All the previous complex (non-K\"ahler) surfaces admit  an invariant SKT structure (i.e.,  induced by a  SKT  structure on the Lie algebra of $L$)
and by \cite{AG, FT} an Inoue surface of type ${\mathcal S}^0$ {admits} an invariant generalized K\"ahler structure.  A $\T^2$-bundle over the Inoue surface of type ${\mathcal S}^0$ was considered in \cite{FT}
in order to construct a $6$-dimensional compact solvmanifold with a non-trivial
generalized K\"ahler structure. A similar construction
can be
done for any of the non-K\"ahler complex homogeneous surfaces, {using} the description of $L$ and $\Theta$ in
\cite{Ha3}. Indeed, in \cite{FT2} it was proved that on any non-K\"ahler compact homogeneous complex surface $X= \Theta \backslash L$
there exists a non-trivial compact $\T^2$-bundle
$M$ carrying a locally conformally balanced SKT metric.

In the sequel we will use tangent Lie algebras associated to $4$-dimensional solvmanifolds to produce examples of SKT metrics and  generalized K\"ahler  structures in higher dimensions.

\begin{ex}
Consider the $4$-dimensional solvable Lie algebra $\g_1$ with structure equations
{\[
\begin{cases}
de^1= e^2 \wedge e^4 \\
de^2=-e^1 \wedge e^4 \\
de^3=e^1 \wedge e^2 \\
de^4=0.
\end{cases}
\]}
{To simplify the notations, we will denote  $\g_1$ by $(e^{24},-e^{14},e^{12},0)$.} On $\g_1$ we define the integrable complex structure $Je^{2i-1}=e^{2i}$ for $i=1,2$ and the SKT inner product $g = \sum_{j = 1}^4 e^j \otimes e^j$. By \cite{Ha3}, $(\g_1,J)$ is the Lie algebra corresponding to the \emph{secondary Kodaira surface}.\\
Since $e_4\notin [\g_1,\g_1]$, we can consider the flat connection $D$ defined in the proof of Proposition \ref{connection}. 
{The tangent Lie algebra $T_D\,\g_1$} has structure equations
\[
(
f^{24},
-f^{14},
f^{12},
0,
-f^{46},
f^{45},
-f^{48},
f^{47}
).
\]
Combining Propositions \ref{connection} and \ref{properties} the induced Hermitian structure $(\tilde J,\tilde g)$ is SKT. 
Moreover, it is easy to verify that the Lie algebra ${\tilde \g}_1 = T_D\,\g_1$ is 3-step solvable and unimodular. The  simply connected Lie group  $\tilde G_1$ with Lie algebra $\tilde {\mathfrak g}_1$ is isomorphic to the semidirect product $\R \ltimes_{\mu} (H_3 \times \C^2)$ where $H_3$ is the real $3$-dimensional Heisenberg Lie group and $\mu$ is the automorphism
$$
\mu(t): (x + i y, u, z_1, z_2) \rightarrow (e^{i\frac{\pi}{2}t} (x + i y), u, e^{i\frac{\pi}{2}t} z_1, e^{i\frac{\pi}{2}t} z_2)
$$
by identifying the matrix
$$
\left ( \begin{array}{ccc} 1&x&u\\ 0&1&y\\ 0&0&1 \end{array} \right )
$$
in $H_3$ with $(x + i y, u) \in \C \ltimes \R$.
Arguing as  in \cite{FT3} it is possible to show that ${\tilde G}_1$ admits a uniform discrete subgroup.

More in general,  {a flat Hermitian connection  $\hat D$} is  on $({\mathfrak g}_1, J, g)$ then it is given with respect to the basis $\{e_i\}$ by 
\begin{equation} \label{Dgenkodsec}
\hat D_{e_i} =0,  \, i = 1,2,3, \quad \hat D_{e_4} = \left (  \begin{array}{cccc}
0&a_{1,2}&a_{1,3}&a_{1,4} \\ -a_{12} & 0 & - a_{1,4} &a_{1,3}\\ - a_{13} & a_{1,4}&0&a_{3,4}\\ - a_{14} & - a_{1,3} & - a_{3,4} & 0   \end{array} \right ),  \, a_{ij} \in \R
\end{equation}
The connection $D$ considered before {coincides} with $\hat D$ when $a_{1,2}=a_{3,4}=1,\ a_{1,3}=a_{1,4}=0$. Note that different choices of the coefficients can lead to {non-isomorphic} Lie algebras. Indeed, when $a_{1,2}=a_{1,3}=a_{1,4}=0$ we obtain that $T_{\hat D}\,\g_1 \cong\R^2\times\mathfrak h$ for a 6-dimensional Lie algebra $\mathfrak h$, so $T_{\hat D}\,\g_1 \ncong \tilde\g_1$.

\end{ex}

\begin{ex}
We start {by} considering the $4$-dimensional solvable Lie algebra 
\[
\g_2= (a e^{14}+b e^{24}, - b e^{14} + a e^{24}, 2a e^{34}, 0),  \quad {a,b \in \R - \{ 0 \},}
\]
{endowed with the two integrable complex structures $J_\pm$ defined by 
\[
J_\pm e^{1}=e^{2} \qquad J_\pm e^3=\pm e^4
\]
and the inner product $g= \sum_{j = 1}^4 e^j \otimes e^j$. By \cite{Ha3}, $(\g_2,J_+)$ corresponds to the \emph{Inoue surface of type  ${\mathcal S}^0$}.} Defining $\omega_\pm(\cdot,\cdot) = g(J_\pm\cdot,\cdot)$ we obtain $d^c_+\omega_+=-d^c_-\omega_-=2ae^{123}$ and $dd^c_+\omega_+=dd^c_-\omega_-=0$, so $(J_\pm,g)$ is a generalized K\"ahler  structure on $\g_2$. We note that $e_4\notin [\g_2,\g_2]$, so applying Proposition \ref{connection} the connection $D$ defined by
\[
\begin{cases}
D_{e_i}=0\quad i=1,2,3 \\
D_{e_4}=J_+
\end{cases}
\]
is flat and {satisfies} $DJ_+=Dg=0$. 
Thus, the induced Hermitian structure $(\tilde J_+,\tilde g)$ on {the} Lie algebra $T_D\,\g_2$ with structure equations 
\[
(a\,f^{14}+b\,f^{24},
-b\,f^{14}+a\,f^{24},
-2a\,e^{34},
0,
-f^{46},
f^{45},
-f^{48},
f^{47}
)
\]
is SKT. Moreover, since $J_+$ and $J_-$ {commute} $DJ_-=0$, hence by Corollary \ref{cor} $(\tilde J_\pm,\tilde g)$ is a generalized {K\"ahler} structure on $T_D\,\g_2$.

Again, $T_D\,\g_2$ is 2-step solvable and unimodular. This generalized K\"ahler Lie algebra was already introduced in \cite{FT} and it was shown that the corresponding simply connected Lie group admits a compact quotient by a discrete subgroup.

More in general, {a flat Hermitian connection  $\hat D$  is expressed   by \eqref{Dgenkodsec}  with respect to the basis $\{e_i\}$}. Moreover, for every choice of the coefficients $a_{1,2},a_{1,3},a_{1,4},a_{3,4}$ we find that $\hat D$ and $J_-$ commute, i.e. $\hat DJ_-=0$. So $(T_{\hat D}\,\g_2,\tilde J_\pm,\tilde g)$ is a GK Lie algebra for every $\hat D$.

\end{ex}

\begin{ex} Consider  the $4$-dimensional nilpotent  Lie algebra
\[
\mathfrak g_3=(0,0, e^{12},0)
\]
endowed with the integrable complex structure $J$ defined by $Je^{2i-1}=e^{2i}$ for $i=1,2$ and the SKT inner product $g = \sum_{j = 1}^4 e^j \otimes e^j$ {corresponding} to the \emph{primary Kodaira surface}. This is a 2-step nilpotent Lie algebra. Indeed $[\g_3,\g_3]=  {{\mbox {span}} <e_3>}$,  {so, in order to generate the action induced by a Hermitian flat connection $D$ on $\R^4$, we need not only $e_4$ (as in the previous example) but also $e_1$ and $e_2$.}
In fact, every connection in the form
\[
D_{e_3} =\mathbf{0}, \qquad D_{e_i} = \left (   \begin{array}{cccc}
0&a_{i,1}&a_{i,2}&a_{i,3} \\ -a_{i,1} & 0 & - a_{i,3} &a_{i,2}\\ - a_{i,2} & a_{i,3}&0&a_{i,4}\\ - a_{i,3} & - a_{i,2} & - a_{i,4} & 0   \end{array} \right )\ \ i=1,2,4 
\]
with $a_{i,j}\in\R$ satisfying the conditions 
\[
\begin{array}{cc}
a_{2,2}(a_{1,1}-a_{1,4})=a_{1,2}(a_{2,1}-a_{2,4})\phantom{\frac{1}{2}}  &\ a_{1,3}\,a_{2,2}=a_{1,2}\,a_{2,3}  \\
a_{3,3}\,a_{2,2}=a_{3,2}\,a_{2,3}\phantom{\frac{1}{2}}  &\  a_{2,2}(a_{3,1}-a_{3,4})=a_{3,2}(a_{2,1}-a_{2,4})
\end{array}
\]
is flat and $Dg=DJ=0$. The tangent Lie algebra $T_D\,\g_3$ is 2-step solvable and unimodular as in the previous cases, but in general it is not nilpotent.

\end{ex}

\section{Taming symplectic forms on $4$-dimensional Lie groups}

We recall that on a complex manifold $(M,J)$ a \emph{taming symplectic form} is a symplectic form $\Omega$ on $M$  such that $\Omega(X,JX)>0$ for every  non-zero vector field $X$ of $M$. {This is equivalent to the existence of an SKT metric $\omega$ such that $\de\omega=\debar\beta$ for a $\de$-closed (2,0)-form $\beta$ (\cite{EFV})}. If $M$ is compact and $(M,J)$ admits a K\"ahler metric, the converse is also true: 
\begin{prop}
Let $(M,J)$ be a compact complex manifold that admits a K\"ahler metric. Then every SKT metric induces a taming symplectic form.
\end{prop}
\begin{proof}
Since $(M,J)$ is compact and admits a K\"ahler metric, the $\de\debar$-lemma holds. Let $\omega$ be the fundamental 2-form of an SKT metric, i.e. $\de\debar\omega=0$. Applying the $\de\debar$-lemma to $\de\omega$ we obtain $\de\omega=\de\debar\gamma$ for some (1,0)-form $\gamma$ on $M$. Then $\de\omega=\debar(-\de\gamma)$ with $\de(-\de\gamma)=0$, so $\omega$ induces a taming symplectic form.
\end{proof}

If $(M,J)$ is compact and 4-dimensional, then it admits a taming symplectic form if and only if it admits a K\"ahler metric \cite[Theorem 1.5]{LiZhang}, so every SKT metric on a compact 4-dimensional {K\"ahler} manifold induces a symplectic form that tames the complex structure. {One can wonder if it also holds for non-compact manifolds. We verify} that is still true in the case of invariant Hermitian structures on $4$-dimensional simply connected Lie groups.

\begin{prop}
Let $(\g,J)$  be a 4-dimensional Lie algebra endowed with a complex structure. Then:
\begin{enumerate}
\item $\g$ admits a taming symplectic form if and only if it admits a K\"ahler metric;
\item if $\g$ admits a K\"ahler metric, every SKT metric induces a taming symplectic form.
\end{enumerate}
\end{prop}
\begin{proof}
It is well known that a non-solvable Lie algebra of dimension 4 is unimodular, so by \cite{LM} it does not admit any symplectic structure. Moreover, an SKT solvable Lie algebra of dimension 4 admits a compact quotient if and only if it is unimodular. Therefore, to study taming symplectic forms on non-compact Lie algebras it is sufficient to consider the non-unimodular case. Using the classification of SKT structures on 4-dimensional Lie algebras in \cite{MS}, for every non-unimodular Lie algebra $\g$ that admits an SKT structure $(J,\omega)$ we provide a $\de$-closed (2,0)-form $\beta$ such that $\de\omega=\debar\beta$ and a $J$-compatible K\"ahler metric $\omega_k$.

Before starting, we note that since $\g$ is 4-dimensional, the space of (2,0)-forms is generated by $\alpha^{12}$, where $\{ \alpha^1,\alpha^2 \}$ is a basis for (1,0)-forms. Thus $\beta$ is $\de$-closed and is in the form $a\,\alpha^{12}$, where $a\in\C$.

Following the notations of \cite{MS}, we study the non-unimodular SKT 4-dimensional Lie algebras:
\begin{itemize}

\item $\R\times\mathfrak r_{3,0} = (0,e^{21},0,0)$. Every SKT metric {has} the form $\omega=e^{12}+e^{34}$ with $Je^1=e^2,\ Je^3=e^4$ and structure equations
\[
(0,0,0,u_1e^{12}+\sqrt{u_1w_1}(e^{14}-e^{23})+w_1e^{34})
\]
where the coefficients are real and $w_1>0,\ u_1\geqslant 0$. We find that $\de\omega=\debar (a\,\alpha^{12})$ with $a=i\,\frac{\sqrt{u_1w_1}}{2w_1}$.

Moreover, the (1,1)-form
\[
\omega_k=ne^{12}+me^{34}+m\,\frac{\sqrt{u_1w_1}}{w_1}(e^{14}-e^{23})
\]
with the conditions $n>m\,\frac{u_1}{w_1},\ m>0$ is closed and positive, so it is a K\"ahler metric with respect to $J$.

\

\item $\mathfrak{aff}_\R\times\mathfrak{aff}_\R = (0,e^{21},0,e^{43})$. Every SKT metric {has} the form $\omega=e^{12}+e^{34}+t(e^{13}+e^{24})$ with $Je^1=e^2,\ Je^3=e^4$ and structure equations
\[
(0,0,x_1e^{12}+x_3(e^{14}-e^{23})+y_2e^{34},u_1e^{12}+u_3(e^{14}-e^{23})+v_2e^{34}),
\]
where $de^2$ and $de^4$ are linearly indipendent and the real coefficients satisfy
\[
\begin{array}{lll}
y_{{2}}x_{{1}}-y_{{2}}u_{{3}}+v_{{2}}x_{{3}}-{x_{{3}}}^{2}=0 &&
u_{{1}}v_{{2}}-u_{{1}}x_{{3}}+u_{{3}}x_{{1}}-{u_{{3}}}^{2}=0 \\
u_{{3}}x_{{3}}-y_{{2}}u_{{1}}=0 && 
(u_1-x_3)(v_2+x_3)-(u_3+x_1)(u_3-y_2)=0.
\end{array}
\]
We find that $\de\omega=\debar (a\,\alpha^{12})$ with $a=-\frac{t}{2}+i\,\frac{u_3-y_2}{2(x_3+v_2)}$. More in general, we have that every $d$-closed 3-form is exact because $b_3^{\,\text{inv}}=\dim H^3(\g)=0$.

Moreover, the (1,1)-form
\[
\omega_k=ne^{12}+me^{34}+p(e^{14}-e^{23})
\]
with the conditions $nx_3-mu_1+p(u_3-x_1)=0,\ nm>p^2$ and $m>0$ is closed and positive, so it is a K\"ahler metric with respect to $J$. Note that the conditions above admit a solution for every {choice} of the coefficients $x_1,x_3,u_1,u_3$. Indeed, fixing $p$, the first condition can be written as $nx_3 = mu_1-p(u_3-x_1)$, so we can choose $n$ and $m$ as large as we need in order to satisfy $nm>p^2$.

\

\item $\mathfrak r'_{4,\lambda,0} = (0,\lambda\, e^{21},e^{41},-e^{31})$, with $\lambda>0$. Every SKT metric {has} the form $\omega=e^{12}+e^{34}$ with $Je^1=e^2,\ Je^3=e^4$ and structure equations
\[
(0,x_1e^{12},y_1e^{12}+y_3e^{14},-y_3e^{13})
\]
where the coefficients are real, $y_3\neq0,\ x_1>0$ and $y_1\geqslant 0$. {Under}  these conditions $\lambda=\vert{\frac{x_1}{y_3}}\vert$. We find that $\de\omega=\debar (a\,\alpha^{12})$ with $a=-\frac{y_1}{2(x_1^2+y_3^2)} (x_1+iy_3)$.

Moreover, the (1,1)-form
\[
\omega_k=ne^{12}+me^{34}+m\,\frac{y_1y_3}{x_1^2+y_3^2}(e^{14}-e^{23})-m\,\frac{y_1x_1}{x_1^2+y_3^2}(e^{13}+e^{24})
\]
with the conditions $n>m\,\frac{y_1}{x_1^2+y_3^2}(y_3-x_1),\ n,m>0$ is closed and positive, so it is a K\"ahler metric with respect to $J$.

\

\item $\mathfrak d_{4,2} = (0,2\,e^{21},-e^{31},e^{41}+e^{32})$. Every SKT metric {has}  the form $\omega=e^{12}+e^{34}$ with $Je^1=e^2,\ Je^3=e^4$ and structure equations
\[
(0,x_1e^{12},y_1e^{12}-\frac{1}{2}x_1e^{13},u_1e^{12}+\frac{1}{2}x_1e^{14}-x_1e^{23})
\]
where the coefficients are real and $x_1>0$. We find that $\de\omega=\debar (a\,\alpha^{12})$ with $a=\frac{1}{3x_1} (-y_1+iu_1)$.

Moreover, the (1,1)-form
\[
\omega_k=ne^{12}+me^{34}+m\,\frac{2u_1}{3x_1}(e^{14}-e^{23})
\]
with the conditions $n>m\,\frac{4u_1^2}{9x_1^2},\ m>0$ is closed and positive, so it is a K\"ahler metric with respect to $J$.

\

\item $\mathfrak d'_{4,\lambda} = (0,\lambda\, e^{21}+e^{31},-e^{21}+\lambda\, e^{31},2\lambda\, e^{41}+e^{32})$, with $\lambda>0$. Every SKT metric {has} the form $\omega=e^{12}+e^{34}+t(e^{13}+e^{24})$ with $Je^1=e^2,\ Je^3=e^4$ and structure equations
\[
\left\{
\begin{aligned}
&d {e}^1 = 0 \\
&d {e}^2 = -k(1+q^2)\, {e}^{12}-kqr( {e}^{14}- {e}^{23})-kr^2\, {e}^{34} \\
&d {e}^3 = \frac{z_3q}{r}\, {e}^{12}-\frac{k}{2}\, {e}^{13}+z_3\, {e}^{14} \\
&d {e}^4 = \frac{q}{r}(kq^2+\frac{k}{2})\, {e}^{12}-z_3\, {e}^{13}+(kq^2-\frac{k}{2}) {e}^{14}-kq^2 {e}^{23}+kqr\, {e}^{34},
\end{aligned}
\right.
\]
with $q,r,k\in\R$ such that $q^2+r^2=1,\ r>0$ and $k,z_3\neq0$. {Under} these conditions $\lambda=\vert{\frac{k}{2z_3}}\vert$. We find that $\de\omega=\debar (a\,\alpha^{12})$ with $a=-\frac{t}{2}-i\frac{q}{2r}$. More in general, we have that every $d$-closed 3-form is exact because $b_3^{\,\text{inv}}=\dim H^3(\g)=0$.

Moreover, the (1,1)-form
\[
\omega_k=m\,\frac{1+q^2}{r^2}e^{12}+me^{34}+m\,\frac{q}{r}(e^{14}-e^{23})
\]
with the condition $m>0$ is closed and positive, so it is a K\"ahler metric with respect to $J$.

\

\item $\mathfrak d_{4,\frac{1}{2}} = (0,\frac{1}{2}\,e^{21},\frac{1}{2}\,e^{31},e^{41}+e^{32})$. Every SKT metric {has}  the form $\omega=e^{12}+e^{34}+t(e^{13}+e^{24})$ with $Je^1=e^2,\ Je^3=e^4$ with the structure equations consider in the latter case with the additional condition $z_3=0$. {Like} before, we find that $\de\omega=\debar (a\,\alpha^{12})$ with $a=-\frac{t}{2}-i\frac{q}{2r}$ and that the (1,1)-form
\[
\omega_k=m\,\frac{1+q^2}{r^2}e^{12}+me^{34}+m\,\frac{q}{r}(e^{14}-e^{23})
\]
with the condition $m>0$ is a K\"ahler metric with respect to $J$.

\end{itemize}
\end{proof}

\end{document}